\documentclass[12pt, a4paper, reqno]{amsart}
\usepackage{latexsym,amsmath, amssymb, eucal, amscd, amstext, enumerate, mathrsfs, yfonts, amsfonts,comment, verbatim, color}
\usepackage[plainpages=false,colorlinks,pdfpagemode=None,bookmarksopen,linkcolor=blue,citecolor=blue,urlcolor=blue]{hyperref}
\usepackage{cite}

\def\tto{\;{\lower 1pt \hbox{$\rightarrow$}}\kern -10pt
\hbox{\raise 2pt \hbox{$\rightarrow$}}\;}

\newtheorem{Theorem}{Theorem}[section]
\newtheorem{Proposition}[Theorem]{Proposition}
\newtheorem{Lemma}[Theorem]{Lemma}
\newtheorem{Corollary}[Theorem]{Corollary}
\theoremstyle{definition}
\newtheorem{Definition}[Theorem]{Definition}
\theoremstyle{remark}
\newtheorem{Remark}[Theorem]{Remark}
\newtheorem{Example}[Theorem]{Example}
\newtheorem{question}[Theorem]{Question}
\begin{document}

\title[Fixed point theorems of various  nonexpansive actions]{Fixed point theorems of various  nonexpansive actions of semitopological semigroups on weakly/weak* compact convex sets}

\author[Muoi]{Bui Ngoc Muoi}
\address[Bui Ngoc Muoi]{Department of Applied Mathematics, National Sun Yat-sen University, Kaohsiung, 80424, Taiwan, and Department of Mathematics, Hanoi Pedagogical University 2, Vinh Phuc, Vietnam.}
\email{\tt buingocmuoi@hpu2.edu.vn, muoibn@mail.nsysu.edu.tw}

\author[Wong]{Ngai-Ching Wong}
\address[Ngai-Ching Wong]{Department of Applied Mathematics, National Sun Yat-sen University, Kaohsiung, 80424, Taiwan; Department of Healthcare Administration and Medical Information, and Center of
Fundamental Science, Kaohsiung  Medical University, 80708 Kaohsiung, Taiwan.}
\email{\tt wong@math.nsysu.edu.tw}

\thanks{Corresponding author: B. N. Muoi, E-mail: buingocmuoi@hpu2.edu.vn, muoibn@mail.nsysu.edu.tw}

\dedicatory{In memory of  Wataru Takahashi (Jan.\ 22, 1944 --- Nov.\ 19, 2020)}

\begin{abstract}
Let $S$ be a right reversible semitopological semigroup, and let $\operatorname{LUC}(S)$ be the space of
left uniformly continuous functions on $S$.
Suppose that  $\operatorname{LUC}(S)$ has a left invariant mean.
Let $K$ be a weakly compact convex subset of a Banach space not necessarily with normal structure. We show that there always exists a common fixed point for any
jointly weakly continuous and super asymptotically nonexpansive action of $S$ on $K$.
Several variances involving the weak* compactness, the RNP, the distality of $K$ and/or the left reversibility of $S$ are also provided.
\end{abstract}



\maketitle

\section{Introduction}
Let $K$ be a non-empty convex subset of a Banach space $E$. Let $T: K\to K$ be a \emph{nonexpansive} map, namely $\|Tx-Ty\|\leq\|x-y\|$ for all $x, y$ in $K$.
Schauder \cite{Schauder30} shows that $T$ has a fixed point if $K$ is norm compact. Kirk \cite{Kirk65} shows that
 $T$ has a fixed point if $E$ is
reflexive  and $K$ is weakly compact with normal structure.
Nevertheless, Alspach \cite{Alspach81} gives an example in which $K$ is weakly compact without normal structure, and $T$ has no fixed point.

The fixed point theorems for actions of a left amenable semigroup were first investigated by Day \cite{Day61} as an extension of the results for commutative family of maps in \cite{Kakutani38,DeMarr63}. In \cite[Problem 4]{Lau1976}
Lau  raised a question about whether the left amenability of a semitopological semigroup $S$ implies the following fixed point property.

\begin{quote}
	$\mathbf{(F_{w^*})}$ Every jointly weak* continuous nonexpansive action of $S$ on a weak* compact convex subset
	$K$ of a dual Banach space has a common fixed point.
\end{quote}

\noindent By embedding a Banach space into its double dual space, the fixed point property $\mathbf{(F_{w^*})}$ implies that

\begin{quote}
	$\mathbf{(F_{w})}$ every jointly weakly continuous nonexpansive action of $S$ on a weakly
	compact convex subset $K$ of a Banach space has a common fixed point.
\end{quote}

 Some partial answers for this question can be found in \cite{LauZhang08,LauZhang12,Mit68,Taka81,SalameII}. 
 In particular, it is stated in \cite{Salame2020} that if $S$ is  $\sigma$--$\operatorname{ELA}$, that is, $S$ is 
  $n$-extremely left amenable for some positive integer $n$, then $\mathbf{(F_{w^*})}$ and $\mathbf{(F_{w})}$ hold 
  without any nonexpansiveness assumption.

\begin{question}
  What can we say if the action of a left amenable semitopological semigroup is \emph{not nonexpansive}?
\end{question}

We can construct an action of a commutative (and thus amenable) discrete semigroup on a compact convex set without common fixed point.
Indeed, Boyce \cite{Boyce69} showed that there are two commutative continuous functions $f, g: [0,1]\to[0,1]$, i.e. $f(g(x))=g(f(x))$ for all $x\in[0,1]$, with no common fixed point.
Therefore, in general, the assumption of nonexpansiveness can not be totally dropped.
 For a resolve, Holmes and Lau \cite{HolLau71} considered
asymptotically nonexpansive actions (see section 2 for definitions) on norm-compact convex sets, while Amini, Medghalchi and Naderi
\cite{AMN2016JMAA} considered pointwise eventually nonexpansive actions on weakly compact convex sets with normal structure. Under some
conditions, these actions have  common fixed points.

\begin{question}
 \emph{Without normal structure}, what can we say for asymptotically nonexpansive like actions on weakly/weak* compact convex sets?
\end{question}

In this paper, we provide a positive answer for the {super asymptotically nonexpansive} actions (see Section
\ref{section2} for definitions).
We  show in Theorem \ref{mainThm}
that if $S$ is right reversible and left amenable then it holds a fixed point property similar to $\mathbf{(F_{w})}$
for super asymptotically nonexpansive actions. When we assume further that the domain $K$ of
 the action is norm-separable, we show in  Theorem \ref{mainthm2} that
 it holds a fixed point property similar to $\mathbf{(F_{w^*})}$. We also provide some results about the existence of a common fixed point
 for an action of any reversible semigroup. In Section \ref{section4}, motivated by recent results of Wi\'{s}nicki \cite{WiS2020}, we
 establish fixed point theorems involving the Radon-Nikod\'{y}m property or the distality.
 In Section \ref{section5}, applying the results in previous sections, we establish  fixed point theorems for
 commutative pointwise eventually
 nonexpansive mappings on weakly/weak* compact convex sets.

\section{Preliminaries}\label{section2}
A \emph{semitopological semigroup} $S$ is a semigroup with a Hausdorff topology such that the product is separately continuous,
i.e., for each fixed $t\in S$, both the maps $s\mapsto ts$ and $s\mapsto st$ from $S$ into $S$ are continuous. Let $\operatorname{CB}(S)$ be the Banach space of bounded and continuous real-valued functions on $S$ equipped with the supremum norm.

For each $s\in S$ and $f\in \operatorname{CB}(S)$, we denote by $l_sf$ the \emph{left translation}
of $f$ by $s$, where $l_sf(t)=f(st)$ for all $t\in S$. Let $\operatorname{LUC}(S)$ be the space of   \textit{left uniformly continuous}
functions on $S$, namely those $f\in\operatorname{CB}(S)$ for which the map $s\mapsto l_sf$ from $S$ into $\operatorname{CB}(S)$ is norm
continuous. A bounded linear functional $m$ on $\operatorname{LUC}(S)$ is called a \emph{mean} if $\|m\|=m(1)=1$.
A mean $m$ is called a \textit{left invariant mean}, or $\operatorname{LIM}$ in short, if $m(l_sf)=m(f)$ for all $s\in S$ and all $f\in\operatorname{LUC}(S)$. We call $S$ \emph{left amenable} if $\operatorname{LUC}(S)$ has a $\operatorname{LIM}$.

An \emph{action} of a semitopological semigroup  $S$ on a Hausdorff topological space $K$ is a mapping of $S\times K$ into $K$, denoted by $(s,x) \mapsto s.x$ (or simply $sx$),
such that $(st).x=s.(t.x)$ for all $s, t\in S$ and $x\in K$.  We call the action  \emph{separately} (resp.\ \emph{jointly}) \emph{continuous}
 if the mapping $(s, x)\mapsto s.x$ is separately (resp.\ jointly) continuous. A point $x_0\in K$ is called a \emph{common fixed point} for $S$ if $s.x_0=x_0$ for all $s\in S$.

Recall that a \emph{left} (resp.\ \emph{right}) \emph{ideal} $I$ of a semigroup $S$ is a nonempty subset of $S$ such that $SI\subseteq I$
(resp.\ $IS\subseteq I$).
We say that a left ideal $I$ is \emph{supported by an element $t$} in $S$ if $I\subseteq St$.
In this case, we can set the left ideal $I^t= \{s\in S: st\in I\}$ and write $I=I^t t$.

\begin{Definition}\label{DefSuperAssym}
	An action $S\times K\to K$ of a semitopological semigroup
 $S$ on a subset $K$ of a Banach space is called
 \begin{enumerate}[{\quad 1.}]
	\item \textit{nonexpansive} if
	$\|s.x-s.y\|\leq\|x-y\|$ for all $s\in S$ and $x, y\in K$.
	\item \textit{asymptotically nonexpansive} (see {\cite{HolLau71}}) if for each given $x, y\in K$, there exists a left ideal
	$I_{xy}$ of $S$ such that  
$\|s.x-s.y\|\leq \|x-y\|$ for all $s\in I_{xy}$;
	\item \textit{pointwise eventually nonexpansive} (see {\cite{AMN2016JMAA}}) if for each given  $x\in K$, there exists a left ideal
	$I_x$ of $S$ such that
 $\|s.x-s.y\|\leq \|x-y\|$  for all $s\in I_x$ and all $y\in K$;

    \item \emph{super asymptotically nonexpansive} if for each given $x\in K$ and $t\in S$, there exists a left ideal $I_{x,t}$ of $S$ supported by $t$
 such that 
$\|s.x-s.y\|\leq \|x-y\|$ for all $s\in I_{x,t}$ \text{and all} $y\in K$.
\end{enumerate}
By taking $I_x^t = \{s\in S: st\in I_{x,t}\}$,
the definition of a super asymptotically nonexpansive   action of $S$ on $K$   can be restated that
for each given $x\in K$ and $t\in S$, there exists a left ideal $I_x^t$ of $S$
 such that
$$
\|st.x-st.y\|\leq \|x-y\|  \quad \text{for all $s\in I_x^t$ and all $y\in K$}.
$$
 \end{Definition}

\begin{Remark}\label{RemarkPointwiseEN}
(a)
Pointwise eventually nonexpansive actions are also called  \emph{semi-asymptotically nonexpansive actions}
 in \cite{AMN2016}, and  \emph{strongly asymptotically nonexpansive actions} in \cite{AAR2018}.

(b)
The notion of pointwise eventually nonexpansive action extends the one introduced by Kirk and Xu in \cite{KirkXu08}, in which
a map $T: K\to K$ is called \emph{pointwise eventually nonexpansive} if for each $x\in K$ there exists $N(x)\in\mathbb{N}$ such that \begin{align}\label{eq:PEN}
\|T^nx-T^ny\|\leq\|x-y\|,\quad \forall y\in K \mbox{ and } n\geq N(x).
\end{align}

(c)	Let $\{T_1,\ldots,T_k\}$ be a commutative family of pointwise eventually nonexpansive maps on $K$. Let $S$ be the
 discrete semigroup generated by this family. Then the action of $S$ on $K$ is 
 super asymptotically nonexpansive. Indeed, for each $x\in K$,
 there exist $N_1(x),\ldots,N_k(x)\in\mathbb{N}$ satisfying \eqref{eq:PEN}
 for  $T_1, \ldots, T_k$, respectively.
 For  any $T_0= T_1^{p_1}\cdots T_k^{p_k}\in S$, set $q_i^{T_0}(x)= \max\{p_i, N_i(x)\}$ for $i=1,\ldots, k$.
Consider the left ideal $I_{x, T_0}=\{T_1^{n_1}\cdots T_k^{n_k}: n_i\geq q_i^{T_0}(x), i=1,\ldots k\}$ of $S$ supported by $T_0$.
Then for each $T\in I_{x,T_0}$ and $y\in K$, we have $\|Tx-Ty\|\leq\|x-y\|$.

\end{Remark}

We have the following implications for the above mentioned actions of semitopological semigroups:
\begin{quote}
	nonexpansiveness $\implies$ super asymptotic nonexpansiveness $\implies$ pointwise eventual nonexpansiveness $\implies$ asymptotic nonexpansiveness.
\end{quote}
As seen in Examples \ref{Ex1} and \ref{Ex2} below, these implications can be strict. However,
as seen in  Proposition \ref{remark210}, when the semigroup $S$ is  compact and right reversible, all three asymptotic nonexpansiveness coincide.

\begin{Example}[based on {\cite[Example 3.3(ii)]{AMN2016}}]\label{Ex1}
	Let $K$ be the closed unit ball in $\mathbb{R}^2$ and $f:[-1,1]\to[-1,1]$ be
 given by $f(x)=x^3$. Consider two maps $T_1$ and $T_2$ from $K$ into $K$ defined by
 $$T_1(x_1,x_2)=(f(x_2),0)\quad \mbox{ and }\quad T_2(x_1,x_2)=(0,f(x_1)).$$
	Clearly, $T_1^n=0$ and $T_2^n=0$ for all $n\geq 2$. Consider the non-commutative discrete semigroup $S$ generated by $T_1$ and $T_2$, that is
	$$S=\{0, T_1, T_2, (T_1T_2)^n, (T_2T_1)^n, T_2(T_1T_2)^n, T_1(T_2T_1)^n: n\in\mathbb{N}\}.$$
	Since $f$ is not nonexpansive, so are $T_1$ and $T_2$.  Hence the action of $S$ on $K$,
 defined by $(T,(x_1,x_2))\mapsto T(x_1,x_2)$, is not nonexpansive.

 However,
 for each $T\in S$ we can find an element $T'\in S$ such that $T'T=0$. Thus the left ideal $\{0\}=(ST')T$ is supported by any $T$ in $S$.
 Therefore, the action of $S$ on $K$ is super asymptotically nonexpansive.
 We also see that $(0,0)$ is a common fixed point of the action.
\end{Example}

\begin{Example}[{based on \cite[Example]{HolLau71}}]\label{Ex2}
	Let $K=\left\lbrace (r,\theta): 0\leq r\leq 1,\, 0\leq \theta<2\pi\right\rbrace$ be the closed unit disk in $\mathbb{R}^2$
	in polar coordinates and the usual Euclidean norm.
Define two continuous mappings $f, g$ from $K$ into $K$ by
	\begin{equation*}
	f(r,\theta)=(r/2,\theta)\quad\text{and}\quad g(r,\theta)=(r,2\theta\ ({\operatorname{mod} 2\pi})).
	\end{equation*}
	Let $S$ be the discrete semigroup generated by $f$ and $g$ under composition.
	Any left ideal $I$ of $S$ must have the form
	$I=\left\lbrace f^ng^m: n\geq n_0, m\geq m_0\right\rbrace$, for some $n_0, m_0\in\mathbb{N}_0$, where $\mathbb{N}_0 =\{0,1,2,\ldots\}$.

	An action of $S$ on $K$ is given by
	$$
	f^ng^m(r,\theta)=(\frac{r}{2^n},2^m\theta\; (\operatorname{mod} 2\pi)).
	$$
As seen in \cite{HolLau71}, the action is asymptotically nonexpansive. However, it is not super asymptotically nonexpansive.
To see this, choose
	$x =(1,0)\in K$, $t =f\in S,$
	and any left ideal $I_0=\left\lbrace f^ng^m: n\geq n_0, m\geq m_0\right\rbrace$, together with
 the left ideal $I=I_0 t$  supported by $t$.
For each $n\geq n_0, m\geq m_0$, consider $s=f^ng^m\in I$ and $y_m=(1,\frac{\pi}{2^m})\in K$.  We have
	$$\|f^{n+1}g^m(x)-f^{n+1}g^m(y_m)\|=\|(\frac{1}{2^{n+1}},0)-(\frac{1}{2^{n+1}},\pi)\|=\frac{1}{2^n},$$ and
	$\|x-y_m\|=\left[2(1-\cos\frac{\pi}{2^m})\right]^{1/2}\rightarrow 0$ as $m\rightarrow \infty$.
	Thus, the inequality
$$\|s.t_0.x-s.t_0.y_m\|\leq \|x-y_m\|$$
fails to hold for all $m\geq m_0$ when $n$ is fixed.
\end{Example}

\begin{Remark} (a) The semigroup $S$ in Example \ref{Ex1} is left amenable because the point evaluation at $0$,
defined by $\delta_0(f)=f(0)$, is a $\operatorname{LIM}$ on $\operatorname{CB}(S)$.
	
	(b) Theorem 3.1 in \cite{HolLau71} does not apply to  Example \ref{Ex1}, since the action in Example \ref{Ex1}
 does not satisfy the property
$\mathbf{(B)}$ in its assumptions. The property $\mathbf{(B)}$ says that for each $x\in K$ whenever a net $\{s_\alpha x\}$ converges
to $x$, the net $\{s_\alpha sx\}$ converges to $sx$ for any $s\in S$. However, consider $x=(1,0)\in K$, $s=T_2$ and the sequence
$s_n=(T_1T_2)^n$ in $S$. We see that $(T_1T_2)^n(1,0)=(1,0)$ and $T_2(1,0)=(0,1)$ while $(T_1T_2)^nT_2(1,0)=(0,0)$ for all $n\in\mathbb{N}$.
\end{Remark}

A semitopological semigroup $S$ is called \emph{right} (resp.\ \emph{left}) \emph{reversible} if any two closed left (resp.\ right) ideals of $S$ always intersect. We call $S$ \textit{reversible} if it is both left and right reversible. For example, the semigroup $S$ in Example \ref{Ex1} is reversible.

\begin{Lemma}[see {\cite[Lemma 3.4]{LauZhang12}}]\label{lemmaLeftReversible}
	For every separately continuous action of a left reversible semitopological
	semigroup $S$ on a compact Hausdorff space $K$, there is a non-empty closed subset $A$ of $K$ such that $A\subset s.A$ for all $s\in S$.
\end{Lemma}

Recall that a Hausdorff topological space $X$ is \emph{countably compact}
if every countable open covering of $X$ has a finite subcovering; in other words,
every countable collection of closed sets in $X$ with the finite intersection property has a nonempty intersection.
A topological space $Y$ is called \emph{C-closed} if every countably compact set in $Y$ is closed (see \cite{IN80}).
For example,  first countable and countably compact spaces are regular, and thus they are C-closed spaces by \cite[Proposition 1.4]{IN80}.

\begin{Proposition}\label{remark210}
	Let $S$ be a  countably compact right reversible semitopological semigroup.
	Consider a   separately  weakly continuous and asymptotically nonexpansive action of $S$ on  a  set  $K$ in a Banach space.
	\begin{enumerate}[(a)]
		\item If  $K$ is norm separable, then the action is pointwise eventually nonexpansive.
		\item If $S$ is   C-closed and $K$ is norm separable, then the action is super asymptotically nonexpansive.
		\item If $S$ is   compact, then the action is  super asymptotically nonexpansive.
	\end{enumerate}
\end{Proposition}
\begin{proof}
	(a)	Let $\left\lbrace y_n\right\rbrace_{n=1}^\infty$ be a norm dense subset of $K$.
	Suppose
	an action of $S$ on $K$ is asymptotically nonexpansive. For each $x\in K$ and $n\in\mathbb{N}$,
	there is a left ideal $I_n$ such that
	$$
	\|a.x- a.y_n\| \leq  \|x-y_n\|, \quad\forall a\in I_n.
	$$
	By the separately weak continuity of the action, we can assume that $I_n$ is closed.
	
	By the right reversibility and the countable compactness of $S$, there is $r\in \bigcap_{n} I_n$.
	Consider the left ideal $I_x:=Sr\subseteq \bigcap_{n} I_n$.
	It follows
	$$
	\|s.x- s.y_n\| \leq \|x-y_n\|, \quad\forall s\in I_x,\ \forall n=1,2,\ldots.
	$$
	By the norm denseness of $\left\lbrace y_n\right\rbrace_{n=1}^\infty$ in $K$ and the separately weak continuity of the action, we see that
	$$
	\|s.x- s.y\| \leq \|x - y\|,\quad\forall s\in I_x, \forall y\in K.
	$$
	In other words, the action is pointwise eventually nonexpansive.

	(b)
	Continuing with the arguments in (a), we see that the action is already pointwise eventually nonexpansive.
	Thus, for any $n=1,2,\ldots$, there is a  left ideal $J_n$ of $S$ such that
	$$
	\|s.y_n - s.y\|\leq \|x-y\|, \quad\forall s\in J_n, \ \forall y\in K.
	$$
	By the separate weakly continuity of the action, we can assume  that $J_n$ is closed.
	By the right reversibility and the countable compactness of $S$,
	we have a (nonempty) closed left ideal $J:= \bigcap_n J_n$.
	For any $t\in S$, the left ideal $St$ is countably compact, and thus closed since $S$ is C-closed.
	By the right reversibility of $S$, we have $(St)\cap J$ is nonempty.
	Consequently,
	$J^t =\{s\in S: st\in  J\}$
	is  a (nonempty) left ideal of S.
	We have
	$$
	\|s.(t.y_n) - s.(t.y)\| = \|(st).y_n - (st).y\| \leq \|y_n-y\|, \quad\forall s\in J^t, \ \forall y\in K.
	$$
	Since $\{y_n\}_{n=1}^\infty$ is norm dense in $K$ and the action is separately weakly continuous,
	$$
	\|s.(t.x) - s.(t.y)\|  \leq \|x-y\|, \quad\forall s\in J^t, \ \forall y\in K.
	$$
	Therefore, the action is super asymptotically nonexpansive.
	
	(c) Suppose that $S$ is compact but $K$ is not necessarily norm separable.
	Note that the left ideal $St$ is compact and thus closed for every $t$ in $S$.
	We can go through the same arguments as in (a) and (b), but by considering the whole set $K$ rather than the sequence $\{y_n\}$,
	to  get the desired conclusion.
\end{proof}

Let $A$ be bounded subset of a Banach space $E$. A point $x\in A$ is called \emph{diametral} if $\sup_{y\in A}\|x-y\|=\operatorname{diam}A$,
 where $\operatorname{diam}A=\sup_{a,b\in A}\|a-b\|$. A convex set $K$ of $E$ is said to have \emph{normal structure} if each bounded,
 convex subset $A$ of $K$ with $\operatorname{diam}A>0$ contains a non-diametral point. Any
 norm compact convex set has normal structure, see
 \cite[Lemma 1]{DeMarr63}, while a weakly (resp.\ weak*) compact convex set might not be so, see \cite{Alspach81,Lim80}.

A closed convex subset $K$ of a Banach space $E$ is said to have the \emph{Radon-Nikod\'{y}m property} ($\operatorname{RNP}$ for short)
 if the following condition is satisfied:
 for any probability space $(\Omega,\mathcal{F},\mu)$ and any $E-$valued measure $F:\mathcal{F}\to E$ which is absolutely continuous with
 respect to $\mu$ such that
 $$
 \{F(B)/\mu(B): B\in\mathcal{F},\mu(B)>0\}\subset K,
 $$
 there is a Bochner integral function
$f\in L^1_E(\Omega,\mathcal{F},\mu)$ such that
$$
F(B)=\int_{B}fd\mu\quad\text{ for every $B\in\mathcal{F}$}.
$$
See, e.g., \cite[Definition 2.1.1]{Bourgin83}.

\begin{Lemma}[see {\cite[Theorem 4.2.13]{Bourgin83}}]\label{LemmaRNP}
	Let $K$ be a weak* compact convex subset of a dual Banach space $E^*$. Then $K$ has  the $\operatorname{RNP}$ if and only if
for any weak* compact subset $Y$ of $K$, the identity map $\operatorname{id}: (Y,\mathrm{wk}^*)\to (Y,\|\cdot\|)$ has a point of continuity.
\end{Lemma}

\begin{Lemma}[based on {\cite[Proposition 3.5]{Megrelishvili98}}]\label{LemmaRNPWeak}
	Let $K$ be a weakly compact subset of a locally convex space $(E,Q)$, where $Q$ is a family of seminorms of $E$
 determining the topology. Let $q\in Q$ and $\varepsilon>0$. Then there is a weakly open subset $U$ of $E$, and a point $x\in K\cap U$ such that $q(x-y)<\varepsilon$ for every $y\in K\cap U$.	
\end{Lemma}

An action $S\times K\to K$ of a semigroup
$S$ on a subset $K$ of the locally convex space $(E,\tau)$ is called \textit{$\tau$-distal} if
$0\notin\overline{\{s.x-s.y: s\in S\}}^{\,\tau}$ for every pair of distinct points $x,y\in K$. We call the action \emph{affine}
if $s.(\lambda x+(1-\lambda)y)=\lambda s.x+(1-\lambda)s.y$ for all $s\in S$, all $x,y\in K$ and all $\lambda\in\left(0,1\right)$.

\begin{Theorem}[see {\cite[Theorem 4.1]{Namioka72}}]\label{ThmAffineDistal}
Let $S\times K\to K$ be a continuous affine action of a semigroup $S$ on a compact convex subset $K$ of a Hausdorff locally convex space $(E,\tau)$. Suppose there exists a nonempty compact subset $A$ of $K$ such that $s.A\subset A$ for all $s\in S$ and the action of $S$ on $A$ is $\tau$-distal. Then there is a common fixed point of $S$ in the convex hull of $A$.
\end{Theorem}

We also consider other classical  function spaces  on $S$ instead of $\operatorname{LUC}(S)$.
Let $\operatorname{AP}(S)$ (resp.\ $\operatorname{WAP}(S)$) be the subspace of $\operatorname{CB}(S)$ consisting of
\textit{almost periodic} (resp.\ \textit{weakly almost periodic}) functions;
namely those $f$ for which the set $\left\lbrace l_sf: s\in S\right\rbrace$ is relatively compact in the norm (resp.\ weak) topology of $\operatorname{CB}(S)$.
In general, we have
$$
\operatorname{AP}(S)\subseteq \operatorname{LUC}(S)\subseteq \operatorname{CB}(S)\quad \text{and}\quad
\operatorname{AP}(S)\subseteq \operatorname{WAP}(S)\subseteq \operatorname{CB}(S).
$$
For the existence of $\operatorname{LIM}$ on these spaces and the associated fixed point properties, the reader can see \cite{Lau73,LauZhang08,Taka92}.

\section{Fixed point theorems assured by the amenability and the reversibility}\label{section3}

\begin{Theorem}\label{mainThm}
	Let $S$ be a right reversible and left amenable semitopological semigroup. Let $K$ be a weakly compact  convex subset of a Banach space. Then every jointly weakly continuous and super asymptotically nonexpansive action of $S$ on $K$ has a common fixed point.

\end{Theorem}

The proof of  Theorem \ref{mainThm} needs several lemmas.
The first one arises from the proof of \cite[Theorem 3.1]{HolLau71}.

\begin{Lemma}\label{lemma1}
	Let $S$ be a right reversible semitopological semigroup.  Assume $S\times K\to K$ is a
separately  continuous action of  $S$
 on a  compact convex subset $K$ of a locally convex space.
Then there exists a subset $L_0$ of $K$ which is minimal with respect to being nonempty,
compact, convex and satisfying the following conditions $\mathbf{(\star 1)}$ and $\mathbf{(\star 2)}$.
	\begin{enumerate}
 	   \item[$\mathbf{(\star 1)}$] there exists a collection $\Lambda = \left\lbrace\Lambda_i: i\in I\right\rbrace$ of  closed subsets of $K$ such that $L_0 =\bigcap \Lambda$, and
    \item[$\mathbf{(\star 2)}$]  for each $x\in  L_0$  there is a left ideal $J_i\subseteq S$ such that $J_i.x\subseteq\Lambda_i$ for each $i\in I$.
	\end{enumerate}
	Furthermore, $L_0$ contains a subset $Y$ that is minimal with respect to being nonempty, compact, and \emph{$S$-invariant},
i.e., $s.Y\subseteq Y$ for all $s\in S$.
\end{Lemma}

The following is a variant of \cite[Lemma 5.1]{LauTaka95}, we sketch a proof here since we need to consult its argument later.
\begin{Lemma}\label{lemma2}
	Let $S\times Y\to Y$ be a jointly weakly continuous action of a semitopological semigroup
 $S$ on a weakly compact subset $Y$ of a normed space. Assume  that $S$ is left amenable.
Assume further that
$Y$ is minimal with respect to being an $S$-invariant, nonempty and weakly compact subset of $Y$. Then $Y$ is norm separable, and
\emph{$S$-preserving}, i.e., $s.Y=Y$ for all $s\in S$.
\end{Lemma}

\begin{proof}
	 For each pair of $y\in Y$ and $f\in \operatorname{C}(Y)$ ($=\operatorname{CB}(Y)$ while $Y$ is the compact space
equipped with the weak topology), define a function $R_yf$ by $R_yf(s)=f(s.y)$.
It can be shown that $R_yf\in \operatorname{LUC}(S)$.
	
	Let $m$ be a $\operatorname{LIM}$ on  $\operatorname{LUC}(S)$.
Define a left invariant linear functional $\psi$ on $\operatorname{C}(Y)$ by $\psi(f)=m(R_yf)$.
 Let $\mu$ be the Radon probability measure on $Y$ defining $\psi$, and
let $Y_0$ be the support of
$\mu$, i.e.,
$$
Y_0=\operatorname{supp}(\mu)=\bigcap\left\lbrace F\subseteq Y: F \ \text{is weakly closed and}\ \mu(F)=1\right\rbrace.
$$
It can be shown that $Y_0$ is $S$-preserving, hence $Y=Y_0$ by the minimality of $Y$. Since every finite Radon measure on a
weakly compact set in a Banach space has a norm separable support (see, e.g., \cite[Theorem 4.3, page 256]{Joram72}), $Y$ is norm separable.
\end{proof}

\begin{Lemma}\label{lem:3}
	Let Y be a norm separable and weakly compact set in a Banach space $E$.
For a super asymptotically nonexpansive action of a right reversible semitopological semigroup $S$ on $Y$,
suppose that $Y$ is minimal with respect to being weakly compact and $S$-invariant. Let $F$ be any nonempty weakly closed subset of $Y$ such that $F\subset s.F$ for all $s\in S$. Then $F$ is norm compact.
In particular, $Y$ is norm compact if $s.Y=Y$ for all $s\in S$.
\end{Lemma}
\begin{proof} We follow the idea in \cite[Lemma 5.2]{LauTaka95} in which a nonexpansive action of $S$ is considered instead.

	Define $N_\varepsilon=\left\lbrace x\in E:\; \|x\|\leq \varepsilon\right\rbrace$ for any given $\varepsilon>0$. Since $Y$ is norm separable,
there exists $\left\lbrace x_i: i\in \mathbb{N}\right\rbrace\subseteq Y$ such that
$Y\subseteq\bigcup\left\lbrace x_i+N_\varepsilon: i\in \mathbb{N}\right\rbrace$.
By the Baire category theorem, there exist $\bar{x}\in\left\lbrace x_i: i\in \mathbb{N}\right\rbrace$ such that $(\bar{x}+N_\varepsilon)\cap Y$
has nonempty interior in $Y$ in the relative weak topology.
Hence, there exist a $z\in Y$ and a weakly open neighborhood $V$ of $0$
such that $(z+V)\cap Y\subseteq(\bar{x}+N_\varepsilon)\cap Y$. We can choose a weak neighborhood $U$ of $0$ such that $U+U\subseteq V$.
Since $U$ also contains a norm open neighborhood of $0$, there exists $\delta>0$ such that
$N_\delta\subseteq U$. By the norm separability of $Y$ again, we can assume, with a new sequence $\{x_i: i\in \mathbb{N}\}$,
that
\begin{align}\label{eq:Yi}
Y=\bigcup\left\lbrace (x_i+N_\delta)\cap Y: i\in\mathbb{N}\right\rbrace.
\end{align}

	By the definition of the super asymptotic nonexpansiveness, for each given $r_0\in S$, there exists a left ideal
$I_1=I_{x_1}^{r_0}$ of $S$ such that $\|sr_0.x_1-sr_0.y\|\leq \|x_1-y\|$ for all  $s\in I_1, y\in Y$.
Since $I_1r_0.x_1$ is $S$-invariant in $Y$, by the minimality of $Y$, its weak
closure must be exactly $Y$. Thus, there exists an $s_1\in I_1$ such that $s_1r_0.x_1\in (z+U)\cap Y$.
Let $r_1=s_1r_0$, we have $r_1.x_1\in (z+U)\cap Y$ and  $\|r_1.x_1-r_1.y\|\leq \|x_1-y\|$ for all  $y\in Y$.
	
	Similarly, there exists a left ideal $I_2=I_{x_2}^{r_1}$ of $S$ such that $\|sr_1.x_2-sr_1.y\|\leq \|x_2-y\|$ for
all  $s\in I_2, y\in Y$. There exists an $s_2\in I_2$ such that
$s_2r_1.x_2\in (z+U)\cap Y$. Let $r_2:=s_2r_1=s_2s_1r_0$,
we have $r_2.x_2\in (z+U)\cap Y$ and $\|r_2.x_2-r_2.y\|\leq \|x_2-y\|$ for all  $y\in Y$.
		By induction, we can choose a sequence $\left\lbrace s_i: i\in \mathbb{N}\right\rbrace$ in $S$ such that
for
$$
r_i=s_is_{i-1}\cdots s_1r_0,
$$
we have
\begin{gather*}
r_i.x_i\in (z+U)\cap Y, \quad\text{and}\\
\|r_i.x_i-r_i.y\|\leq \|x_i-y\|, \quad\forall  y\in Y,\ i\geq 1.
\end{gather*}
	
	For each $y\in (x_i+N_\delta)\cap Y$, we can write
$$
r_i.y=(r_i.y-r_i.x_i)+r_i.x_i
$$
where $r_i.x_i\in (z+U)\cap Y$ and $\|r_i.x_i-r_i.y\|\leq\|x_i-y\|<\delta$. Thus,
$$
r_i.((x_i+N_\delta)\cap Y)\subseteq (z+U+N_\delta)\cap Y\subseteq (z+V)\cap Y.
$$
	We rewrite the action $r.x$ in the form of $L_rx$.
Then $(x_i+N_\delta)\cap Y\subseteq L_{r_i}^{-1}((z+V)\cap Y)$, where $L_{r_i}^{-1}((z+V)\cap Y)$ is weakly open by the weak continuity of the action.
By the weak compactness and \eqref{eq:Yi}, we can cover $Y$ by  finitely many such weakly open sets.  Let
$$
Y=\bigcup_{i=1}^{n}L_{r_i}^{-1}((z+V)\cap Y).
$$
	
	It follows from the super asymptotic nonexpansiveness of the action that
 there exist left ideals $J_i=J_{\bar{x}}^{t_i}$, where
 $$
 t_i=s_{n+1}s_{n}\cdots s_{i+1},\quad\text{for}\ i=1,\ldots, n,
 $$
 such that
 $\|st_i.\bar{x}-st_i.y\|\leq \|\bar{x}-y\|$ for all  $y\in Y, s\in J_i$. Since $S$ is right reversible, there exists $t_0\in\cap_{i=1}^{n}\overline{J_i} $.

	For each  $i$, there exists a net $\left\lbrace s_\lambda\right\rbrace\subseteq J_i$ converging to $t_0$. Hence
$s_\lambda.( t_i.\bar{x})-s_\lambda.( t_i. y)$ converges to $t_0 t_i.\bar{x}-t_0t_i. y$ weakly. Therefore, from the lower continuity of the norm function in the weak topology,
\begin{align}\label{eq:dist-bar}
\|t_0t_i.\bar{x}-t_0t_i.y\|\leq \|\bar{x}-y\| \quad\text{for all}\  y\in Y,\ i=1,\ldots, n.
\end{align}

Since $F\subset s.F$ for all $s\in S$, we have
	\begin{equation*}
	\begin{array}{rl}
	F\subset L_{t_0}L_{r_{n+1}}F\subset L_{t_0}L_{r_{n+1}}Y &=L_{t_0}L_{r_{n+1}}\left\lbrace \bigcup_{i=1}^{n}L_{r_i}^{-1}((z+V)\cap Y)\right\rbrace\\
	&\subseteq \bigcup_{i=1}^{n}\left\lbrace L_{t_0}L_{s_{n+1}\cdots s_{i+1}}((\bar{x}+N_\varepsilon)\cap Y)\right\rbrace\\
	&= \bigcup_{i=1}^{n}\left\lbrace L_{t_0t_i}((\bar{x}+N_\varepsilon)\cap Y)\right\rbrace\\
	&\subseteq \bigcup_{i=1}^{n}\left\lbrace (L_{t_0t_i}\bar{x}+N_\varepsilon)\cap Y)\right\rbrace\\
	&\subseteq \bigcup_{i=1}^{n}\left\lbrace L_{t_0t_i}\bar{x}+N_\varepsilon\right\rbrace.
	\end{array}
	\end{equation*}
The second last inclusion above follows from \eqref{eq:dist-bar}. This proves that the norm closed set $F$ can be covered by a finite $\epsilon$-net for
any $\epsilon>0$.  Hence, $F$ is norm compact.
	\end{proof}

Together with above lemmas and motivated by \cite[Theorem 3.1]{HolLau71}  and \cite[Theorem 4.2]{AAR2018}, we are ready to prove Theorem \ref{mainThm}.  Note that, if the subset $Y$ in Lemma \ref{lem:3} is known to be
 convex then we can apply \cite[Lemmas 2.5]{SGF2018} for a shorter proof.
 However, at the current stage, we do not have the convexity of $Y$.

\begin{proof}[Proof of Theorem \ref{mainThm}]
By Lemmas \ref{lemma2} and \ref{lem:3}, the nonempty
 $S$-preserving subset $Y$ given in Lemma \ref{lemma1} is separable and norm compact in the Banach space $E$. Consequently, the norm topology and the weak topology  agree on $Y$.
If  $Y$ contains exactly one point  then we are done. Otherwise, let
$$
r=\mbox{diam} (Y)=\sup\left\lbrace\|x-y\|: x, y\in Y\right\rbrace.
$$
By DeMarr's Lemma \cite[Lemma 1]{DeMarr63}, there is an element $u\in\overline{\operatorname{conv}}(Y)$
such that
$$
r_0=\sup\left\lbrace\|u-y\|: y\in Y\right\rbrace<r.
$$
Let $0<\varepsilon<r-r_0$.
Let $L_0$ and $\left\lbrace\Lambda_i: i\in I\right\rbrace$ be given in Lemma \ref{lemma1}.
For each $\Lambda\in \left\lbrace\Lambda_i: i\in I\right\rbrace$,
set
\begin{align*}
N_{\varepsilon,\Lambda}&= \bigcap_{y\in Y} \{x\in \Lambda : \|x-y\|\leq  r_0+\varepsilon\}
\intertext{and}
N_0&=\bigcap\left\lbrace N_{\varepsilon, \Lambda_i}:  \ i\in I\right\rbrace= L_0\cap \bigcap_{y\in Y}\bar{B}[y, r_0+\varepsilon],
\end{align*}
where $\bar{B}[y, \delta]$ stands for the norm closed ball centered at $y$ of radius $\delta$.

We show that $N_0$  satisfies $\mathbf{(\star 1)}$ and $\mathbf{(\star 2)}$. Indeed, every $N_{\varepsilon,\Lambda_i}$ is weakly compact.
Thus $N_0$ is a weakly compact subset of $L_0$, and contains $u$.
For each $x\in N_0$ and  $i\in I$, there exists a left ideal $I\subseteq S$ such that $I.x\subseteq \Lambda_i$.
By the super asymptotic nonexpansiveness of the action, for each $t\in S$ there exists a left ideal $I_{x}^t$
such that  $\|st.y-st.x\|\leq \|y-x\|$ for all  $y\in K$ and $s\in I_{x}^t$.
By the right reversibility of $S$, there exists a  $t_0\in\overline{I}\cap\overline{I_{x}^t t}$. Since $\Lambda_i$ is weakly closed, $St_0.x\subseteq \Lambda_i$.
   Consider a net $s_\lambda\in I_{x}^t$ such that $s_\lambda t\to t_0$.
From  $\|ss_\lambda t.y-ss_\lambda t.x\|\leq \|y-x\|\leq r_0+\varepsilon$ for all $\lambda$, $y\in Y$ and $s\in S$,
we have $\|st_0.y-st_0.x\|\leq \|y-x\|\leq r_0+\varepsilon$.
Since $Y\subset st_0.Y$, we have $\|y'- st_0.x\|\leq r_0+\varepsilon$ for all $y'\in Y$.
In other words, there exists a left ideal $J=St_0$ of $S$
such that $J.x\subseteq N_{\varepsilon, \Lambda_i}$.  Consequently, the nonempty  weakly compact
convex subset $N_0$ of $L_0$ also satisfies conditions $\mathbf{(\star 1)}$ and $\mathbf{(\star 2)}$.

By the minimality of $L_0$, we have
$Y\subseteq L_0= N_0 \subseteq \bigcap_{y\in Y}\bar{B}[y, r_0+\varepsilon]$.
This gives us a contradiction that $\operatorname{diam}(Y) \leq r_0+\epsilon < r$.
Therefore, $Y$ contains a unique point and it is the common fixed point for the action of $S$ on $K$.
\end{proof}

\begin{Remark}
	The converse of Theorem \ref{mainThm} is not true in general. In fact, by Proposition \ref{sep+reversibility} below, we will see that for any separable and reversible semitopological semigroup $S$, any jointly weakly continuous and super asymptotically nonexpansive action of $S$ on a weakly compact convex set $K$ has a common fixed point. However, $S$ is not necessarily left amenable. For example, take $S$ to be the free group of two generators.
\end{Remark}

 Lau and Zhang \cite[Theorem 6.2]{LauZhang12} established that  a left amenable semitopological semigroup $S$
  has the following fixed point property.

\begin{quote}
 $\mathbf{(F_{w^*,sep})}$ Let $K$ be a weak* compact convex and norm-separable subset of a dual Banach space. Then every  jointly weak* continuous and nonexpansive  action   of  $S$ on  $K$ has a common fixed point.
\end{quote}
Replacing the assumption of nonexpansiveness with the weaker one of
super asymptotic nonexpansiveness, but together with the  right reversibility of the semigroup, we obtain the following result.

\begin{Theorem}\label{mainthm2}
Let  $S$ be a right reversible and left amenable semitopological semigroup. Let $K$ be a weak* compact convex and norm-separable subset of a dual Banach space. Then every jointly weak* continuous and super asymptotically nonexpansive action of $S$ on $K$ has a common fixed point.

\end{Theorem}
\begin{proof}	
It follows similarly as in proving Theorem \ref{mainThm}, but with the norm-separability coming from the assumption, and noticing that
Lemma \ref{lem:3} is also valid for the weak* compact case.
	\end{proof}

\begin{Remark} Since the support of a finite Radon measure on a weak* compact set in a dual Banach space \emph{may not}
 be norm separable,  the conclusion of Lemma \ref{lemma2} about the norm separability of $Y$ may not hold for the weak* compact case.
\end{Remark}

The following result supplements {\cite[Theorem 3.4]{LauZhang08}}.

\begin{Theorem}\label{APWAP}
	Let $S$ be a right reversible semitopological semigroup.
	\begin{itemize}
		\item[(i)] Assume $\operatorname{AP}(S)$ has a $\operatorname{LIM}$.  Let $K$ be a weakly compact (resp.\ weak* compact
and norm-separable) convex subset of a Banach (resp.\ dual Banach) space. Then every separately weakly (resp.\ weak*) continuous, equicontinuous and super asymptotically nonexpansive action of $S$ on $K$ has a common fixed point.
		\item[(ii)] Assume $\operatorname{WAP}(S)$ has a $\operatorname{LIM}$.  Let $K$ be a weakly compact (resp.\ weak* compact
and norm-separable) convex subset of a Banach (resp.\ dual Banach) space. Then every separately weakly (resp.\ weak*) continuous, quasi-equicontinuous and super asymptotically nonexpansive action of $S$ on $K$ has a common fixed point.
		
	\end{itemize}
\end{Theorem}
\begin{proof}
	These are direct consequences of
\cite[Lemma 3.1]{Lau73}, \cite[Lemma 3.2]{LauZhang08}, and the proofs of Theorems \ref{mainThm} and  \ref{mainthm2}.
	\end{proof}

\begin{Corollary}\label{corr.Normal}
	Let $S$ be a  semitopological semigroup as well as a normal topological
space.  Assume that $\operatorname{CB}(S)$ has an invariant mean. Let $K$ be a weakly compact  (resp.\ weak* compact  and
	norm-separable) convex subset of a Banach (resp.\ dual Banach) space.
Then every separately weakly (resp.\ weak*) continuous and super asymptotically nonexpansive action of $S$ on $K$ has a common fixed point.
\end{Corollary}
\begin{proof}
It is  known that if $S$ is normal and $\operatorname{CB}(S)$ has a right invariant mean then $S$ is right reversible. The assertion now follows from Theorems \ref{mainThm} and \ref{mainthm2}.
\end{proof}

\begin{Remark}
	The only place we need the joint continuity of the action in
	Theorems \ref{mainThm} and \ref{mainthm2}  is where we derive that $R_yf$ belongs to $\operatorname{LUC}(S)$ for each $f\in \operatorname{C}(Y)$ and $y\in Y$.  Then we can construct a $\operatorname{LIM}$ $\psi$ of $\operatorname{C}(Y)$.
	For Theorem \ref{APWAP} and Corollary \ref{corr.Normal}, we need only separate continuity since other assumptions there suffice to ensure that
	such a $\operatorname{LIM}$ $\psi$ exists for the stated function spaces on $S$.
\end{Remark}

Without any amenability assumption,
we consider in the following  fixed point properties of reversible semitopological semigroups. Borzdyński and Wi\'{s}nicki
\cite[Theorem 3.5]{BW2014} established that any commutative (and thus left amenable and reversible) semigroup has the fixed point property
$\mathbf{(F_{w^*})}$.
For a discrete semigroup, the left amenability implies
the left reversibility (see \cite[page 2549]{LauZhang08}), while in general it might not be the case.
The following two results supplement Theorems \ref{mainThm} and \ref{mainthm2}.  The key point in their proofs is to bypass Lemma \ref{lemma2}.

\begin{Proposition}\label{coroLeftRever}
	Let $S$ be a reversible semitopological semigroup. Let $K$ be a weak* compact convex and norm-separable subset of a dual Banach space.
 Then every separately weak* continuous and super asymptotically nonexpansive action of $S$
		on $K$  has a common fixed point.
\end{Proposition}
\begin{proof}
	By  Lemma \ref{lemma1}, there is a subset $L_0$ of $K$ which is minimal with respect to being nonempty, weak* compact,
	convex, and satisfying conditions $\mathbf{(\star 1)}$ and $\mathbf{(\star 2)}$
where the weak* topology is involved. Moreover, $L_0$ contains
	a subset $Y$ that is minimal with respect to being nonempty, weak* compact and $S$-invariant.
	
By Lemma \ref{lemmaLeftReversible}, there is a non-empty weak* closed subset $F$ of $Y$ such that $F\subset s.F$ for all $s\in S$.
 Applying Lemma \ref{lem:3}, we see that $F$ is compact.
	The remaining parts now follow similarly as in proving Theorem \ref{mainThm} where the set $Y$ is replaced by its norm compact subset $F$.
\end{proof}

\begin{Proposition}\label{sep+reversibility}
	Let $S$ be a separable and reversible semitopological semigroup. Let $K$ be a weakly compact  convex subset of a Banach space. Then every separately weakly continuous and super asymptotically nonexpansive action of  $S$ on $K$ has a common fixed point.
\end{Proposition}
\begin{proof}
	By Lemma \ref{lemma1}, we establish a subset $Y$ of $L_0\subset K$ that is minimal with respect to being nonempty, weakly compact
	and $S$-invariant.
	Following an idea in \cite[Lemma 3.3]{LauZhang08}, we show that $Y$ is norm separable. Indeed, for any
	fixed $y\in Y$, we have $Sy=\{s.y: s\in S\}$ is an $S$-invariant subset of $Y$. Thus, its weak closure $\overline{Sy}^{\,\mathrm{wk}}$ must be exactly $Y$. Assume $S$ contains   a countable dense subset $S_c$. Since the action is separately weakly
	continuous, $\overline{Sy}^{\,\mathrm{wk}}=\overline{S_cy}^{\,\mathrm{wk}}$. Moreover,
	$\overline{\operatorname{conv}}^{\,\mathrm{wk}}(Sy)=\overline{\operatorname{conv}}^{\,\mathrm{wk}}(S_cy)
=\overline{\operatorname{conv}}^{\,\|\cdot\|}(S_cy)$ by Mazur's Theorem.
It follows that $Y=\overline{\operatorname{conv}}^{\,\mathrm{wk}}(Sy)$ is norm separable. 
From Lemma \ref{lemmaLeftReversible}, the left reversibility of $S$ ensures that there is a  weakly compact subset $F$ of $Y$ satisfying $F\subset s.F$ for all $s\in S$.
	The remaining parts follow similarly as in the proof of Theorem \ref{mainThm}.
\end{proof}	

\begin{Remark}
 In Example \ref{Ex1}, all assumptions of Theorem \ref{mainThm} and Proposition \ref{sep+reversibility} are satisfied. Thus we can conclude that the action there has a fixed point. Indeed, the action has a common fixed point $(0,0)$.
\end{Remark}

Motivated by \cite[Page 2550]{LauZhang08}, we call a semitopological semigroup $S$ \emph{strongly reversible} if there is a family of separable
and
reversible subsemigroups $\{S_\alpha: \alpha\in I\}$ of $S$ such that $S=\bigcup_{\alpha\in I} S_\alpha$ and for each $\alpha_1, \alpha_2\in I$ there is an $\alpha_3\in I$ such that $S_{\alpha_1}\cup S_{\alpha_2}\subset S_{\alpha_3}$. Here, the topology on each $S_\alpha$ is the
subspace topology inherited
 from $S$.
 Obviously, if $S$ is strongly reversible then it is reversible. Follow \cite[Lemma 5.2]{LauZhang08}, if $S$ is metrizable and reversible then it is strongly reversible.

\begin{Corollary}
	Let $S$ be a strongly reversible semitopological semigroup. Let $K$ be a weakly compact  convex subset of a Banach space. Then every separately weakly continuous and super asymptotically nonexpansive action of  $S$ on $K$ has a common fixed point.
\end{Corollary}
\begin{proof}
 Suppose $S=\bigcup_{\alpha\in I} S_\alpha$ where $S_\alpha$ is separable and reversible. For each $\alpha\in I$, consider the sub-action $\overline{S_\alpha}\times K\to K$ of $\overline{S_\alpha}$ on $K$. It follows from the proof of Proposition \ref{sep+reversibility} that there is a norm separable subset $Y_\alpha$ of $K$ that is minimal with respect to being nonempty, weakly compact and $\overline{S_\alpha}$-invariant. Follow the arguments in proving \cite[Lemma 5.3]{LauZhang08} and Lemma \ref{lemmaLeftReversible}, we can show that there is a nonempty weakly compact subset $F_\alpha$ of $Y_\alpha$ such that $F_\alpha\subset s.F_\alpha$ for all $s\in\overline{S_\alpha}$. By Lemma \ref{lem:3}, $F_\alpha$ is norm compact. As that verified in the proof of Theorem \ref{mainThm}, the common fixed points set $F_\alpha=\{x\in K: s.x=x\; \mbox{ for all } s\in \overline{S_\alpha}\}$ is nonempty. Since the action is weakly continuous, $F_\alpha$ is weakly closed.  Indeed, for any net $\{x_\alpha\}$ in $F_\alpha$ weakly converging to $x$ in $K$,
 we have $s.x = \lim_\alpha s.x_\alpha = \lim_\alpha x_\alpha = x$ and thus $x\in F_\alpha$. Since for each $\alpha_1, \alpha_2\in I$ there is an $\alpha_3\in I$ such that $S_{\alpha_1}\cup S_{\alpha_2}\subset S_{\alpha_3}$, the family $\{F_\alpha: \alpha\in I\}$ has the finite intersection property.
 By the weak compactness of $K$, the set of common fixed points $F(S)=\bigcap_{\alpha\in I} F_\alpha$  of $S$ is nonempty.
\end{proof}

\section{Fixed point theorems involving the Radon-Nikod\'{y}m property or the distality}\label{section4}

Recently, Wi\'{s}nicki \cite{WiS2020} provided an extension for the Ryll-Nardzewski's Theorem.
It is about the existence of a common fixed point for a nonlinear action of a semigroup on a weakly compact convex set
in a locally convex space. Following his idea, we establish in this section some extensions of the results in
Section \ref{section3} as well as the results in \cite{WiS2020}, for the asymptotically nonexpansive
type actions of right reversible semigroups.

The main idea is to replace the norm-separability assumption in \cite[Theorem 6.2]{LauZhang08} and in Theorem \ref{mainthm2}
with the Radon-Nikod\'{y}m property or the norm-distality, and
 to derive the norm compactness of an $S$-invariant subset of $K$ in the action.

\begin{Theorem}\label{ThmRNP}
	Let  $S$ be a right reversible and left amenable semitopological semigroup. Let $K$ be a weak* compact convex subset of a dual Banach space with the $\operatorname{RNP}$. Then every jointly weak* continuous and pointwise eventually nonexpansive action of $S$ on $K$ has a common fixed point.
\end{Theorem}

\begin{proof}
	 By Lemmas \ref{lemma1} and \ref{lemma2},
there is an $S$-invariant Radon probability measure $\mu$
such that the  support $Y=\operatorname{supp}\mu$
 is minimal with respect to being a nonempty
 $S$-invariant weak* compact subset of $K$.  Below, we follow
 the approach in \cite[Theorem 3.1]{WiS2020}, see also
 \cite[Theorem 4.2]{WiS2019Arxiv}, in which   nonexpansive actions are considered instead.

 By Lemma \ref{LemmaRNP}, there exists an $x\in Y$ such that for each $\varepsilon>0$, there is a weak* open neighbourhood $U$ of $x$ such
 that $\|x-y\|<\varepsilon$ for all $y\in U\cap Y$. Hence
	
	\begin{equation}\label{eqSec4}
	\delta:=\mu(\{y\in Y: \|x-y\|<\varepsilon\})\geq\mu(U\cap Y)>0.
	\end{equation}
Since the action is pointwise eventually nonexpansive, there is a left ideal $I$ of $S$ such that
$$\|s.x-s.y\|\leq\|x-y\|,\quad \forall s\in I, \forall y\in Y.$$
	Therefore, for each $s\in I$, we have
$$
\{y\in Y: \|x-y\|<\varepsilon\}\subset L_s^{-1}\{z\in Y: \|s.x-z\|<\varepsilon\},
$$
where $L_s(y):=s.y$ for $y\in K$.
	Since $\mu$ is invariant,
	\begin{equation*}
	\begin{array}{rl}
	\mu(\{y\in Y: \|s.x-y\|<\varepsilon\})&=\mu(L_s^{-1}\{y\in Y: \|s.x-y\|<\varepsilon\})\\
	&\geq \mu(\{y\in Y: \|x-y\|<\varepsilon\})=\delta>0.
	\end{array}
	\end{equation*}

	Since $Ix$ is $S$-invariant, by the minimality of $Y$, we have $Y=\overline{Ix}^{\,\mathrm{wk}^*}$.
We shall see that there are only finitely many elements $s_1,\ldots,s_k$ in $I$ such that $\|s_ix-s_jx\|\geq 2\varepsilon$ for any $i\neq j$.
	In fact, all $Y_i=\{y\in Y: \|y-s_ix\|<\varepsilon\}$ are pairwise disjoint subsets of $Y$, and $\mu(Y_i)\geq\delta$ for all $i$.
Now the fact $\mu(Y)=1$ ensures that at most finitely many of such elements exist. In other words, $Ix$ is totally bounded in norm,
and hence $\overline{Ix}^{\,\|\cdot\|}$ is norm-compact.
Since the identity map from $(Y,\|\cdot\|)$ into $(Y,\mathrm{wk}^*)$ is continuous,
$Y=\overline{ Ix }^{\,\mathrm{wk}^*}=\overline{ Ix }^{\,\|\cdot\|}$ is norm-compact. The remaining   follows similarly as the proof of
Theorem \ref{mainThm}.	
\end{proof}

 The following result supplements Theorem \ref{ThmRNP}, and applies to the case when we do not have the left amenability in stock.

\begin{Theorem}\label{ThmRNPDis}
	Let $S$ be a right reversible semitopological semigroup. Let $K$ be a weak* compact convex subset of a dual Banach space with the $\operatorname{RNP}$. Then every separately weak* continuous, pointwise eventually nonexpansive and norm-distal action of $S$ on $K$ has a common fixed point.
\end{Theorem}
\begin{proof}
By  Lemma \ref{lemma1}, there is a subset $L_0$ of $K$ which is minimal with respect to being nonempty, weak* compact,
convex,
 and satisfying conditions $\mathbf{(\star 1)}$ and $\mathbf{(\star 2)}$ where the weak* topology is involved. Moreover, $L_0$ contains
a subset $Y$ that is minimal with respect to being nonempty, weak* compact and $S$-invariant.

We are going to construct an $S$-invariant Radon probability measure $\mu$
on $K$ with respect to
the weak* topology. As in the proof of \cite[Theorem 3.1]{WiS2020}, the norm-distality of the action together with the minimality of $Y$ implies that the action is weak*-distal.

Let ${\operatorname{C}(Y)}$ be the space of continuous functions on $Y$. Let $\mathbf{P}(Y)$ be the weak* compact convex set of all means on $\operatorname{C}(Y)$. Consider an action of $S$ on $\mathbf{P}(Y)$ given by $s.\mu=l_s^*\mu$,
 where $l_s$ is the left translation operator by $s$, and $\left\langle l_s^*\mu,f\right\rangle=\left\langle \mu,l_sf\right\rangle$ for all $s\in S$ and $f\in {\operatorname{C}(Y)}$.
 Let $\phi: Y\to \mathbf{P}(Y)$ be the isometric natural embedding,  $x\mapsto\phi(x)=\hat{x}$, defined by $\hat{x}(f)=f(x)$ for all
 $f\in {\operatorname{C}(Y)}$.
 Then, we obtain an action of $S$ on $\phi(Y)$ by defining
  $s.\phi(x)=\phi(s.x)$ for all $s\in S$ and $x\in Y$.

  Since the action of $S$ on $Y$ is
 weak*-distal, so is  the action of $S$ on $\phi(Y)$.
 It follows from Theorem \ref{ThmAffineDistal} that there is a
 common fixed point $\mu$ of $S$ in $\mathbf{P}(Y)$.
 In other words, $\mu$ is an $S$-invariant Radon probability measure on $Y$ with respect to the weak* topology.
As in the proof of Theorem \ref{ThmRNP}, we see that $S$ has a common fixed point in $K$.
\end{proof}

For the super asymptotically (resp.\ pointwise eventually) nonexpansive actions on a weakly compact convex subset of a locally convex space, we obtain the following fixed point properties  without assuming neither $\operatorname{RNP}$ nor norm-separability.

Let $(E,Q)$ be a locally convex space where $Q$ is a family of seminorms determining the topology. An action $S\times K\to K$ is said to be
\begin{enumerate}	
	\item  \emph{pointwise eventually $Q$-nonexpansive} if for each $x\in K$ and $q\in Q$, there exists a left ideal $I=I(x,q)$ of $S$ such that $q(s.x-s.y)\leq q(x-y)$ for all $s\in I$ and all $y\in K$;
	
	\item \emph{supper asymptotically $Q$-nonexpansive} if for each $x\in K$, $t\in S$ and $q\in Q$, there exists a left ideal $I=I(x,t,q)$ of $S$ supported by $t$ such that $q(s.x-s.y)\leq q(x-y)$ for all $s\in I$ and all $y\in K$.
\end{enumerate}
See \cite{MuoiWong2020Frechet} for more discussions.

The following result is an extension of Theorem \ref{mainThm} for pointwise eventually nonexpansive actions.
\begin{Proposition}\label{Prop4.3}
	Let $S$ be a right reversible and left amenable semitopological semigroup. Let $K$ be a weakly compact  convex subset of a locally convex space $(E,Q)$. Then every jointly weakly continuous and pointwise eventually $Q$-nonexpansive action of $S$ on $K$ has a common fixed point.
\end{Proposition}
\begin{proof}
	Using Lemma \ref{LemmaRNPWeak} and arguing as in proving Theorem \ref{ThmRNP}, we obtain a similar inequality as \eqref{eqSec4}
in which the norm is replaced by a seminorm in $Q$.  We can then
 derive that $Y$ is $Q$-compact. Finally, with an argument similar to the one proving Theorem \ref{mainThm}, see also \cite[Theorem 2.14]{MuoiWong2020Frechet}, we will arrive at the conclusion.
\end{proof}

\begin{Theorem}\label{ThmDistal}
	Let $S$ be a separable and right reversible semitopological semigroup. Let $K$ be a weakly compact convex subset of a locally convex space $(E,Q)$. Then every separately weakly continuous, super asymptotically $Q$-nonexpansive and $Q$-distal action of $S$
		on $K$ has a common fixed point.
\end{Theorem}

\begin{proof}
	By Lemma \ref{lemma1}, there is a subset $Y$ of $K$, that is minimal with respect to being nonempty,
 weakly compact and $S$-invariant.  As in proving Proposition \ref{sep+reversibility}, we can show that $Y$ is $Q$-separable.
	
	As in proving \cite[Theorem 4.1]{WiS2020}, the $Q$-distality of the action together with the minimality of $Y$ ensures
 that the action is indeed weakly-distal. Following the proof of Theorem \ref{ThmRNPDis}, we can show the existence of an
  $S$-invariant Radon probability measure $\mu$ defined on $Y$ such that $Y=\operatorname{supp}(\mu)$ and $sY=Y$ for all $s\in S$.
  Arguing as in proving Lemma \ref{lem:3}, but with the  seminorms in $Q$ replacing the norm,
  we see that $Y$ is $Q$-compact. The remaining now follows as in the proof of Theorem \ref{mainThm}.
\end{proof}

\section{Fixed point theorems for pointwise eventually nonexpansive mappings}\label{section5}

In the following, applying the result in previous sections, we establish fixed point theorems for a finite commutative
family of continuous maps on weakly/weak* compact convex sets which are pointwise eventually nonexpansive.

\begin{Corollary}\label{CoroKirkXu}
	Let $K$ be a non-empty weakly compact convex subset of a Banach space. Let $\{T_1,\ldots,T_k\}$ be a commutative family of weakly continuous and pointwise eventually nonexpansive maps on $K$. Then they have a common fixed point.
\end{Corollary}
\begin{proof}
	Let $S$ be the discrete semigroup generated by this family.
As seen in Remark \ref{RemarkPointwiseEN}(c), the action of $S$ on $K$ is super asymptotically nonexpansive. Moreover, since $S$ is commutative, it is reversible and left amenable.
	From Theorem \ref{mainThm}, as well as Proposition \ref{sep+reversibility}, $S$ has a common fixed point. Hence $\{T_1,\ldots,T_k\}$ has a common fixed point in $K$.	
\end{proof}

\begin{Corollary}
	Let $K$ be a weak* compact convex subset of a dual Banach space with the $\operatorname{RNP}$. Let $\{T_1,\ldots,T_k\}$ be a commutative family of weak* continuous and pointwise eventually nonexpansive maps on $K$. Then they have a common fixed point.
\end{Corollary}
\begin{proof}
Let $S$ be the semigroup   generated by this finite family.
	From Theorem \ref{ThmRNP}, the canonical
 action of $S$ on $K$  has a common fixed point. Therefore, the finite family has a common fixed point in $K$.
\end{proof}
Since every weak* compact convex and norm-separable subset has the $\operatorname{RNP}$, we have the following result.
\begin{Corollary}\label{CoroKirkXuWeakStar}
	Let $K$ be a weak* compact convex and norm-separable subset of a dual Banach space. Let $\{T_1,\ldots,T_k\}$ be a commutative family of weak* continuous and pointwise eventually nonexpansive maps on $K$. Then they have a common fixed point.
\end{Corollary}
\begin{proof}
	This is a consequence of Theorem \ref{mainthm2}, as well as Remark \ref{RemarkPointwiseEN}(c) and Proposition \ref{coroLeftRever}.
\end{proof}
A map $T: K\to K$ on a subset $K$ of a locally convex space $(E, Q)$ is called \emph{eventually nonexpansive} if for each $x, y\in K$ and each seminorm $q\in Q$, there exists an $n(x,y,q)\in\mathbb{N}$ such that $q(T^nx-T^ny)\leq q(x-y)$ for all $n\geq n(x,y,q)$.

\begin{Corollary}
	Let $K$ be a compact convex subset of a locally convex space $(E,Q)$. Let $\{T_1,\ldots,T_k\}$ be a commutative family of continuous and eventually nonexpansive maps on $K$. Then they have a common fixed point.
\end{Corollary}
\begin{proof}
	This follows Theorem 3.1 in \cite{HolLau71}. Noting that for commutative semigroups, the property \textbf{(B)} is always   satisfied.
\end{proof}

We end this paper with an open problem about possible extensions of our results.
Under some conditions,
we establish that a super asymptotically nonexpansive or pointwise eventually nonexpansive action  of
a semitopological semigroup $S$ on weakly/weak* compact convex sets has a common fixed point.

\begin{question}\label{question:2}
Do we have similar results as   Theorems \ref{mainThm}, \ref{mainthm2}, \ref{ThmRNP}, \ref{ThmRNPDis}, \ref{ThmDistal}, and Propositions \ref{coroLeftRever}, \ref{sep+reversibility} for asymptotically nonexpansive actions?
\end{question}

From Proposition \ref{remark210}, if $S$ is compact and right reversible then the asymptotic nonexpansiveness coincides
with the super asymptotic nonexpansiveness. Hence
the question has an affirmative answer in this case.

\section*{Acknowledgment}
The authors would like to thank K. Salame for valuable suggestions and comments.
B. N. Muoi is supported by Hanoi Pedagogical University 2, Vietnam grant HPU2.UT-2021.03, and both authors are supported by Taiwan MOST grants 110-2811-M-110-520, 108-2115-M-110-004-MY2 and 110-2115-M-110-002-MY2.

\end{document}